\documentclass[a4paper]{article}

\fontsize{10.5pt}{12pt}

\title{Lower bounds on the maximum number of non-crossing acyclic graphs}
\date{}

\author{Clemens Huemer\thanks{Departament de Matem\`atica Aplicada IV, Universitat Polit\`ecnica de Catalunya, Barcelona, Spain. E-mail address: \texttt{clemens.huemer@upc.edu}.}, Anna de Mier\thanks{Departament de Matem\`atica Aplicada II, Universitat Polit\`ecnica de Catalunya, Barcelona, Spain.  E-mail address: \texttt{anna.de.mier@upc.edu}. }}

\usepackage{epsfig,graphicx,a4wide}
\usepackage{amsfonts,amsmath,amsthm,amssymb}
\newtheorem{theorem}{Theorem}[section]
\newtheorem{proposition}[theorem]{Proposition}
\newtheorem{lemma}[theorem]{Lemma}
\newtheorem{corollary}[theorem]{Corollary}

\begin{document}

\maketitle
\begin{abstract}

This paper is a contribution to the problem of counting geometric graphs on point sets. More concretely, we look at the maximum numbers of non-crossing spanning trees and forests. We show that the so-called double chain point configuration of $N$ points has $\Omega(12.52^N)$ non-crossing spanning trees and $\Omega(13.61^N)$ non-crossing forests. This improves the previous lower bounds on the maximum number of non-crossing spanning trees and of non-crossing forests among all sets of $N$ points in general position given by Dumitrescu, Schulz, Sheffer and T\'oth. Our analysis relies on the tools of analytic combinatorics, which enable us to count certain families of forests on points in convex position, and to estimate their average number of components. A new upper bound of $O(22.12^N)$ for the number of non-crossing spanning trees of the double chain is also obtained.
\end{abstract}

\section{Introduction}

\bigskip

A geometric graph  on a point set $S$ (throughout, $S$ has no three collinear points) is a graph with vertex set $S$ and whose edges are straight-line segments with endpoints in $S$. A geometric graph is called {\emph{non-crossing}} (nc- for short) if no two edges intersect except at common endpoints. Counting nc-geometric graphs is a prominent problem in combinatorial geometry, since Ajtai et al.~\cite{ajtai} showed in 1982 that there exists a constant $c>0$ such that the number of nc-geometric graphs on sets $S$ of $N$ points is bounded from above by $O(c^N).$ 
Here we focus on non-crossing acyclic graphs, that is, spanning trees and forests. In~\cite{hoffmann} it is proved that no set $S$ of $N$ points has more than $O(141.07^N)$ nc-spanning trees. The maximum number of nc-spanning trees (among all sets of $N$ points) is very likely much smaller.

The point set with most nc-spanning trees known so far is the so-called {\emph{double chain}}. The double chain of $N=2n$ points consists of two sets of $n$ points each, one forming a convex chain and one forming a concave chain. We denote them by upper  and lower chain. Furthermore, each straight-line defined by two points from the upper chain leaves all the points from the lower chain on the same side, and reversely; see Figure~\ref{fig:lbtree}. We sometimes will refer to the left and to the right side of the double chain, where the point of the double chain with smallest abscissa is on the left and the point with largest abscissa is on the right.

Counting nc-geometric graphs on the double chain was initiated by Garc\'ia et al.~\cite{GNT} who proved that it has $\Theta^*(8^N)$ triangulations{\footnote{We use the $O^*$-, $\Theta^*$-, and $\Omega^*$-notation to describe the asymptotic growth of the number of geometric graphs as a function of the number $N$ of points, neglecting polynomial factors. If a class of nc-graphs has $\Theta^*(c^N)$ elements on $N$ points, we say that $c$ is the \emph{growth constant} of the class.}}, $\Omega(9.35^N)$ nc-spanning trees and $\Omega(4.64^N)$ nc-polygonizations, where the latter bound also is the current best lower bound on the maximum number of nc-polygonizations among all sets $S$ of $N$ points. The lower bound for the number of nc-spanning trees of the double chain was subsequently improved to $\Omega(10.42^N)$~\cite{D99} and to  $\Omega(12.0026^N)$~\cite{D11}. 
Our main contribution is an improvement of this bound to $\Omega(12.52^N)$, which  also improves our previous bound of $\Omega(12.31^N)$ presented in~\cite{eurocomb}. We also give a new lower bound of $\Omega(13.61^N)$ for the maximum number of nc-forests among all sets $S$ of $N$ points.  Along the way we also obtain the asymptotic growth for the number of nc-spanning trees of a point set similar to the double chain, the so-called {\emph{single chain}}. In this set the upper chain is replaced by a single point; see Figure~\ref{fig:singlechain}. The single chain was considered in~\cite{AOSS}, where the number of pseudo-triangulations was counted. We prove that it has $\Theta^*(9.5816\ldots^{n})$ nc-spanning trees.

The paper is organized as follows. In Section~\ref{sec:lbound} we describe a construction on the double chain that produces $\Omega(12.52^N)$ nc-spanning trees, and we include a variation for generating nc-forests. 
The key counting ingredient for obtaining our lower bounds is being able to enumerate certain classes of forests on points in convex position, and, more crucially, to estimate their average number of components. This is achieved through the methods of analytic combinatorics, with an spirit and techniques similar to those in~\cite{ncconfigsFN}, and it is the content of Section~\ref{sec:gfs}. 
A curious consequence of these calculations for points in convex position, is that the growth constant of nc-forests is one more than the growth constant of nc-forests that have no isolated vertex. In Section~\ref{sec:growth} we give 
an explanation of this fact and show that it also applies to other families of nc-graphs on points in convex position. We also prove that the growth constant of nc-forests of the double chain is at least one more than the growth constant of those forests without isolated vertices, and from this fact we prove our lower bound of $\Omega(13.61^N)$ nc-forests of the double chain. 
Finally, in Section~\ref{sec:upper} we provide a new upper bound of $O(22.12^N)$ for the number of nc-spanning trees of the double chain.


\section{Lower bounds}\label{sec:lbound}

In this section we give lower bounds for the numbers of spanning trees and forests of the double chain. For spanning trees, the result is the following.

\begin{theorem}\label{thm:lbound}
The double chain on $N$ points has $\Omega(12.52^N)$ non-crossing spanning trees.
\end{theorem}

We  next describe a family of trees that gives the desired bound. Our construction depends on some parameters that need to be maximized later.

For any spanning tree of the double chain on $N=2n$ points, the vertices on the upper and lower chains induce two forests $F_U$ and $F_L$ on a set of $n$ points in convex position; there are also some edges with one endpoint on each chain (the \emph{interior} edges). The first restriction is that we consider trees where only one vertex in each component of $F_U$ is incident to interior edges; this vertex will be called the \emph{mark} of the component and we call $F_U$ a \emph{marked forest}.

Of the several interior edges that are incident to a mark $v$, let $e_v$ be the rightmost one. The second restriction we impose is that  for each component $C$ of $F_L$ that is not incident to an edge $e_v$, there is a unique edge joining this component to a mark $m_C$ in $F_U$, and this edge has as endpoint in $F_L$ the leftmost vertex $v_C$ in $C$. Moreover, $m_C$ is as to the right as possible.

The set of edges $M_1=\{e_v: v \mbox{ is a mark}\}$ induces a forest. Observe that  this forest is uniquely determined by the leftmost edge in each component (assuming the set of marks is known). The set of these edges is denoted $M_2$; note that $M_2$ is a matching.

See Figure~\ref{fig:lbtree} for an example of a spanning tree of the double chain satisfying the conditions above.

\begin{figure}
\begin{center}
\includegraphics[height=4cm]{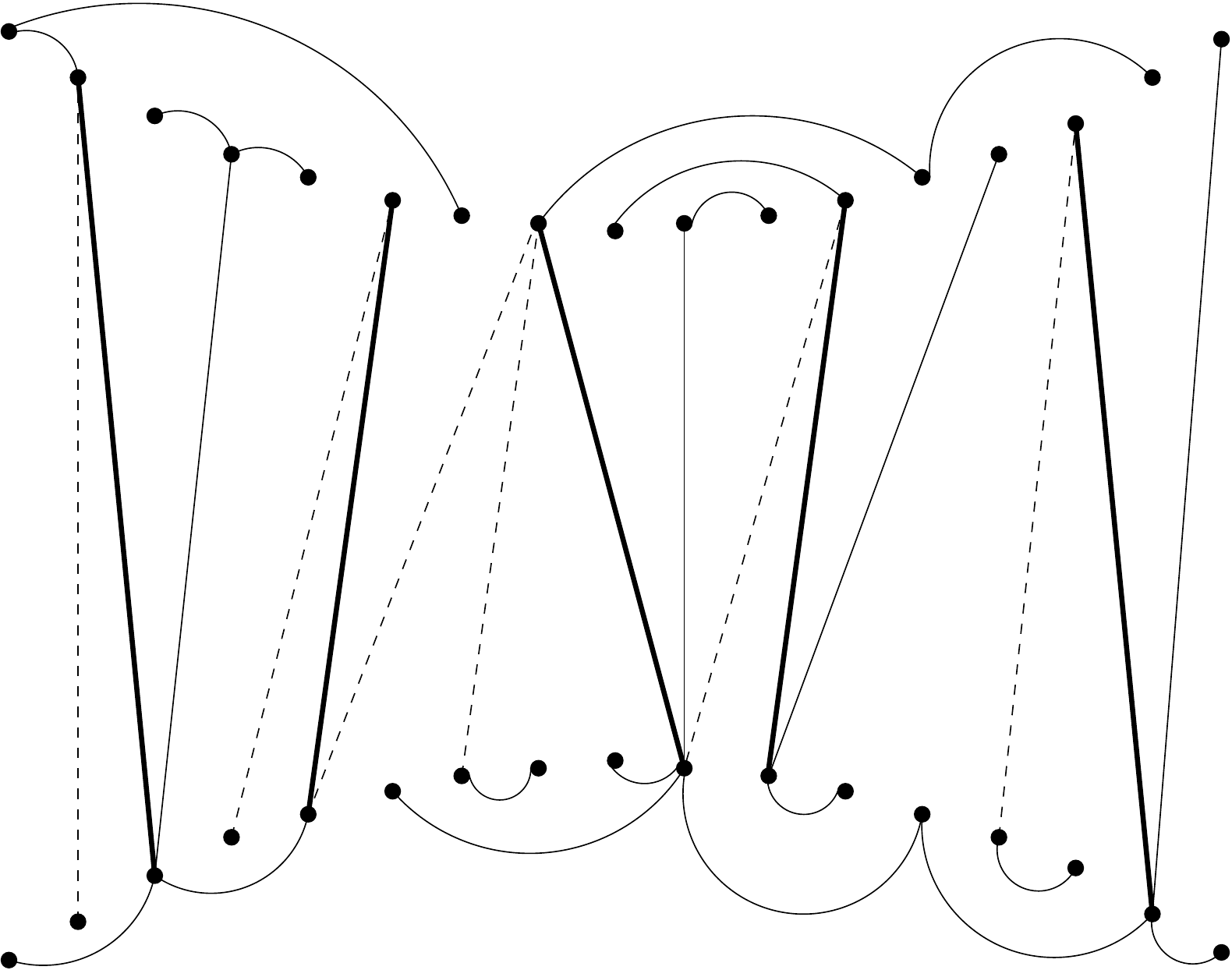}
\end{center}
\caption{A spanning tree of the double chain. The bold edges are the edges of $M_2$ and the thin continuous interior edges are the other edges of $M
_1$. The edges in $F_U$ and $F_L$ are drawn as arcs.
}
\label{fig:lbtree}
\end{figure}

We claim that the following data are enough to construct one such  tree with $|M_2|=k$:

\begin{itemize}
\item[a1)] A marked nc-forest $F_U$ with at least $k$ components on a set of $n$ points in convex position,
\item[a2)] a subset of $k$ of the marks in $F_U$,
\item[a3)] a nc-forest $F_L$ on a set of $n$ points in convex position, and
\item[a4)] a subset $M_L$ of $k$ vertices in $F_L$.
\end{itemize}

Indeed, it suffices to do the following:

\begin{itemize}
\item[t1)] Match the $k$ marks from a2) with the $k$ vertices in $F_L$ (corresponding to the edges of $M_2$);
\item[t2)] join the other marks in $F_U$ to the leftmost visible vertex in $M_L$ (corresponding to the edges of $M_1\backslash M_2$);
\item[t3)] for each component of $F_L$ that has no vertex in $M_L$, take its leftmost vertex and join it to the rightmost visible mark;
\item[t4)] if the result is not connected, for each component $C$ let $l_C$ be its leftmost mark. Label the connected components as $C_1,\ldots,C_r$ in such a way that the marks $l_{C_1},\ldots, l_{C_r}$ are ordered from left to right. Now, for $i\geq 2$, join $l_{C_i}$ to the rightmost visible vertex of $M_L$ in $C_{i-1}$.
\end{itemize}

Note that in this way we do not generate all the spanning trees satisfying the restrictions imposed above, but that all the trees that are generated are different, which is enough for our purposes.

We next count in how many ways we can choose the forests and subsets in items a1)--a4), which essentially amounts to enumerating forests in points in convex position. As stated in the proposition below, the number of nc-forests on $n$ points in convex position where each component has a mark is $\Theta^*(9.5816\ldots^{n})$, and for $n$ sufficiently large, at least $40\%$ of them have  $0.2237n$ or more components each. However, there is another way of choosing the forest $F_U$ from a1) that gives us more choice. First we select $\ell$ vertices out of the $n$ in the upper chain and we mark all of them, and then on the remaining $n-\ell$ points, we choose a marked forest $F'_U$ such that all components in this forest have at least two vertices.

Of the estimates below, (i) is well-known~(\cite{ncconfigsFN}) and the other two will be proved in Section~\ref{sec:gfs}.

\begin{proposition}\label{prop:numberforests}
\begin{itemize}
\item[(i)] The number of non-crossing forests on $n$ points in convex position is $bn^{-3/2}\omega_F^n(1+O(1/n))$, where $\omega_F=8.2246\ldots$ and $b$ is a constant. 
\item[(ii)] The number of marked non-crossing forests on $n$ points in convex position is $cn^{-3/2} \omega_M^n(1+O(n^{-1/2}))$, where $\omega_M=9.5816\ldots$ and $c$ is a constant.  Moreover, for $n$ large enough at least $40\%$ of these forests have $0.2237n$ or more components.
\item[(iii)] The number of marked non-crossing forests on $n$ points in convex position such that no component is an isolated vertex is $dn^{-3/2}\omega
_U^n(1+O(n^{-1/2}))$, where $\omega_U=8.5816\ldots$ and $d$ is a constant. Moreover, for $n$ large enough at least $40\%$ of these forests have $0.1332n$ or more components.
\end{itemize}
\end{proposition}


With this knowledge, setting $k=\alpha n$ and $\ell=\beta n$ (for $\alpha,\beta \in (0,1)$ to be determined), and ignoring subexponential terms, we get the following lower bound on the number of nc-spanning trees of the double chain:
$$b(\alpha,\beta)=\binom{n}{\beta n}8.5816^{n-\beta n}\binom{0.1332(n-\beta n)+\beta n}{\alpha n} 8.2246^{n}\binom{n}{\alpha n}.$$

Note that when choosing the $k=\alpha n$ marks from the upper forest $F_U$ we are actually not using all the available marks, but only the ones that come from isolated vertices and the first $0.1332(n-\beta n)$ marks of the marked forest $F'_U$.

Using the binary entropy function $H(x)=-x \log_2(x) - (1-x)\log_2(1-x)$ and Stirling's formula, we can estimate the binomial coefficient ${{\epsilon n}\choose{\delta n}}$ as $2^{\epsilon H(\frac{\delta}{\epsilon})n},$ ignoring again subexponential terms.

It thus is enough to maximize
\begin{align*}
e(\alpha,\beta)=& H(\beta)+(1-\beta)\log_2(8.5816)+(0.1332(1-\beta)+\beta)H(\frac{\alpha}{0.1332(1-\beta)+\beta})+\\ &\log_2(8.2246)+H(\alpha).
\end{align*}

The values
$\alpha=0.267, \beta=0.267$ give
$\Omega^*(2^{(7.293063\ldots /2)N})=\Omega^*(12.5232\ldots^N)$ nc-spanning trees on the double chain on $N$ vertices, thus proving Theorem~\ref{thm:lbound}.\\

We remark that the bound of $\Theta^*(9.5816\ldots^{n})$ marked forests on $n$ points in convex position
gives the number of nc-spanning trees of the single chain. This point set has triangular convex hull and all but one point $p$ of the set are in convex position; see Figure~\ref{fig:singlechain}. To count the number of nc-spanning trees of this set, notice that the deletion of $p$ gives a marked nc-forest on the lower chain (the mark of each component being the vertex that was adjacent to $p$).

\begin{figure}[t]
	\centering
		\includegraphics{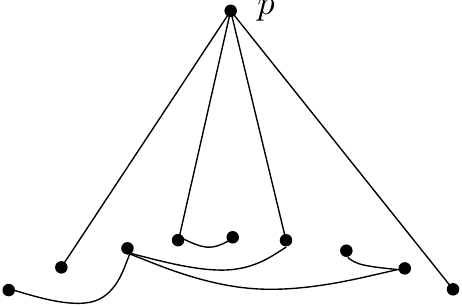}
	\caption{A spanning tree of the single chain, obtained from a marked forest on the lower chain. }
	\label{fig:singlechain}
\end{figure}


\begin{corollary}
The single chain on $N$ points has  $\Theta^*(9.5816\ldots^N)$ non-crossing spanning trees.
\end{corollary}

\bigskip

Using a similar construction, we next prove a lower bound of $\Omega(13.40^N)$ for the number of nc-forests of the double chain. This is not our main result on the number of nc-forests, as in Section~\ref{sec:growth} we further improve this bound to $\Omega(13.61^N)$. We nevertheless present  it here  to show how far the ideas in this section can be pushed in the case of forests, and because the improvement in Section~\ref{sec:growth} needs as a first step the number  of nc-forests of a particular kind, which will be bounded using a variation of the following construction.

\begin{theorem}\label{thm:lboundforest}
The double chain on $N$ points has $\Omega(13.40^N)$ non-crossing forests.
\end{theorem}

The restriction of a nc-forest of the double chain to each of the chains gives two forests $F_U$ and $F_L$. Again, we consider only forests where in each component of $F_U$ there is at most one vertex incident with interior edges. Thus, $F_U$ is a forest where some components have a mark and some others do not. 
As before, let $e_v$ be the rightmost interior edge incident with a mark $v$, let $M_1=\{e_v:v \mbox{ is a mark}\}$, and let $M_2$ be the set consisting of the leftmost edge in each component of $M_1$. We impose the further restriction that no edge $e_v$ is incident with an isolated vertex of $F_L$.

A forest satisfying these conditions can be constructed from the following data.

\begin{itemize}
\item[b1)] A  nc-forest $F_U$ where at least $k$ of its components have a mark,
on a set of $n$ points in convex position,
\item[b2)] a subset of $k$ of the marks in $F_U$,
\item[b3)] a nc-forest $F_L$ on a set of $n$ points in convex position, 
\item[b4)] a subset $M_L$ of $k$ vertices in $F_L$, none of which is isolated, and
\item[b5)] a subset of $m$ vertices of $F_L$ such that all of them are isolated.
\end{itemize}

Indeed, it suffices to do the following:

\begin{itemize}
\item[f1)] Match the $k$ marks from b2) with the $k$ vertices in $F_L$ (corresponding to the edges of $M_2$);
\item[f2)] join the other marks in $F_U$ to the leftmost visible vertex in $M_L$ (corresponding to the edges of $M_1\backslash M_2$);
\item[f3)] join each vertex from the subset given in b5) with the rightmost visible mark of $F_U$.
\end{itemize}

We use the following estimates, also proved in Section~\ref{sec:gfs}.

\begin{proposition}\label{prop:forestforest}
\begin{itemize}
\item[(i)] The number of nc-forests with no isolated vertex on a set of $n$ points in convex position is $b n^{-3/2} \omega_L^n(1+O(n^{-1/2}))$, where $\omega_L=7.2246\ldots$ and $b$ is a constant.
\item[(ii)] On a set of $n$ points in convex position, the number of nc-forests where some of the connected components that are not isolated vertices bear a mark is $c n^{-3/2} \omega_S^n(1+O(n^{-1/2}))$, where $\omega_S=9.8643\ldots$ and $c$ is a constant. Moreover, at least $40\%$ of these forests have $0.1106n$ or more connected components that bear marks.
\end{itemize}
\end{proposition}

As we did for trees, to choose the forest $F_U$ we first pick $\ell$ vertices that will be isolated and will all bear a mark, and then on the remaining $n-\ell$ vertices we  take a forest  where some of the components that are not isolated are marked. Let $k=\alpha n$, $\ell=\beta n$ and $m=\gamma n$.
Using the estimates from Proposition~\ref{prop:forestforest}, the number of nc-forests of the double chain on $2n$ points is at least
$$f(\alpha,\beta,\gamma)=\binom{n}{\beta n} 9.8643^{n-\beta n} \binom{0.1106(n-\beta n)+\beta n}{\alpha n} \binom{n}{\gamma n}\binom{n-\gamma n}{\alpha n} 7.2246^{n-\gamma n}2^{\gamma n}.$$

We estimate this quantity as before and the values $\alpha=0.235$, $\beta=0.245$ and $\gamma=0.166$ give $\Omega^*(13.4025\ldots^N)$ nc-forests on the double chain on $N$ points, as needed.


\bigskip

Before moving to the proofs of Propositions~\ref{prop:numberforests} and~\ref{prop:forestforest}, observe that, in the notation of those propositions, one has
$$ \omega_F=\omega_L+1, \qquad \omega_M=\omega_U+1.$$
That is, for points in convex position, the growth constant for nc-forests is one more than the growth constant for nc-forests without isolated vertices, and similarly for marked nc-forests.

We are not aware that such a relationship has been observed before, and we give a combinatorial proof of it in Section~\ref{sec:growth}. Thus, if we are only interested in growth rates and not in the exact asymptotic behaviour, one need not carry the calculations explained in Section~\ref{sec:gfs} for the classes that have no isolated vertices. 

Furthermore, we prove in Section~\ref{sec:growth} a related inequality between the growth constants of nc-forests in the double chain, which allows us to improve the bound in Theorem~\ref{thm:lboundforest}.

\section{Non-crossing forests of points in convex position}\label{sec:gfs}

In this section we use generating functions and the techniques of analytic combinatorics to prove Propositions~\ref{prop:numberforests} and~\ref{prop:forestforest}.

Consider a set of $n$ points in convex position, labelled counterclockwise $p_1,\ldots,p_n$; the vertex $p_1$ is called the root vertex.   A systematic study of several classes of non-crossing graphs with generating functions was undertaken by Flajolet and Noy in~\cite{ncconfigsFN}; the results in this section are an extension of theirs using similar techniques (see also the book~\cite{ancombFS} by Flajolet and Sedgewick).

Let $\mathcal{T}$ be a set of nc-trees and let $\mathcal{F_T}$ be the set of those nc-forests such that its connected components belong to $\mathcal{T}$ (by taking as the root of each component the vertex with smallest label and relabelling the other vertices suitably). Let $T(z)$ and $F_T(z)$ be the corresponding generating functions, that is,
$$T(z)= \sum_{n\geq 1} t_nz^n,\qquad
F_T(z)=\sum_{n\geq 0} f_nz^n,$$
where $t_n$ and $f_n$ denote the number of $n$-vertex graphs in $\mathcal{T}$ and $\mathcal{F_T}$, respectively. For technical convenience, we set $f_0=1$ (but $t_0=0$).

We have the following key relation
\begin{equation}\label{eq:tf}
F_T(z)=1+T(zF_T(z)).
\end{equation}
The combinatorial explanation is as follows (see Figure~\ref{fig:tf}). Given a forest in~$\mathcal{F_T}$, let $t_1$ be the connected component that contains the root vertex; this component is of course a tree of~$\mathcal{T}$. Now, the vertices (if any) that lie strictly between any two consecutive vertices of $t_1$ induce a nc-forest, which belongs to $\mathcal{F}_T$. The substitution $zF_T(z)$ in $T(z)$ reflects this (where the term $z$ corresponds to a vertex of $t_1$). Thus, if an  equation for $T(z)$ is known, we immediately get from~(\ref{eq:tf}) an equation for $F_T(z)$. We now follow this scheme and find equations for the generating functions for the different classes of forests under consideration.

\begin{figure}
\begin{center}
\includegraphics[height=3cm]{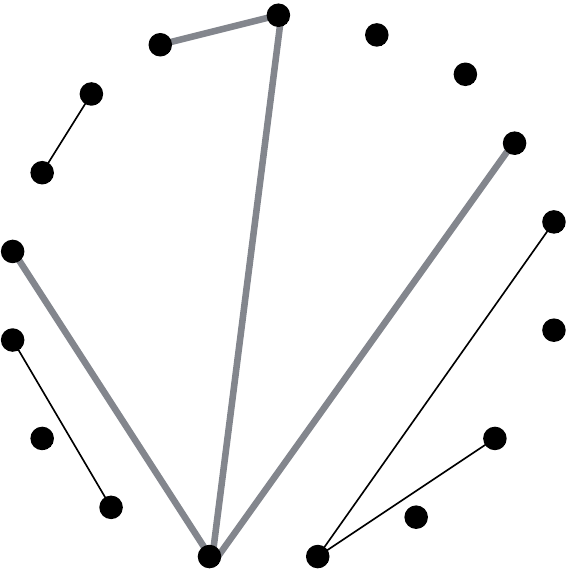}
\end{center}
\caption{A nc-forest decomposes as a nc-tree (in grey) with a (possibly empty) nc-forest between any two of its consecutive vertices.}
\label{fig:tf}
\end{figure}

Let ${\mathcal{T}}_0$ be the set of all nc-trees. It is well-known that the generating function $T_0(z)$  satisfies
\begin{equation}\label{eq:trees}
T_0(z)^3-zT_0(z)+z^2=0.
\end{equation}
Let now $\mathcal{T}^*$ be the class of marked trees, that is, trees with one vertex distinguished; as a tree with $n$ vertices gives rise to $n$ marked trees, we have $T^*(z)=zT_0'(z)$.  By differentiating equation~(\ref{eq:trees}) and eliminating $T_0(z)$ we get an equation for $T_0'(z)$, namely,
\begin{equation} \label{eq:tp}
(27z^2-4z)T_0'(z)^3+(1-6z)T'_0(z)-1+8z=0.
\end{equation}
From this and~(\ref{eq:tf}) we obtain an equation for the generating function for marked forests, $Y=F_{T^*}(z)$:
\begin{equation*}
27zY^4+(8z^3-6z^2-81z-4)Y^3+(5z^2+82z+12)Y^2-(28z+12)Y+4=0.
\end{equation*}
Now let $\mathcal{T}_1$ be the class of nc-trees with more than one vertex; clearly $T_1(z)=T_0(z)-z$. From equations~(\ref{eq:trees}) and~(\ref{eq:tf}) it follows that $Y= F_{T_1}(z)$ satisfies
\begin{equation*}
(1+z)^3 Y^3-(3z^2+7z+3)Y^2+(4z+3)Y-1=0.
\end{equation*}
As for the class $\mathcal{T}^*_1$ of marked trees with at least two vertices, the corresponding generating function is $T_1^*(z) =zT'(z)-z$. 
From~(\ref{eq:tp}) and~(\ref{eq:tf}) we obtain the following equation for $Y=F_{T_1^*}(z)$:
\begin{equation}
27z(1+z)^3Y^4-(83z^3+180z^2+93z+4)Y^3+
(99z^2+106z+12)Y^2-(12+40z)Y+4=0.\label{eq:f1s}
\end{equation}

The last class of forests we need is the one where 
some of the components that are not isolated vertices have a mark. This class is $\mathcal{F}_{\mathcal{T}_0\cup \mathcal{T}_1^*}$. We obtain the corresponding equation for its generating function as above:
\begin{equation*}
27z(1+z)^3Y^4-(20z^3+234z^2+93z+4)Y^3+6(3z+3)(8z+1)Y^2-12(4z+1)Y+4=0.
\end{equation*}

Once an algebraic equation for $F_T(z)$ is known, it is usually routine to obtain an asymptotic estimate of the coefficients of $F_T(z)$.
 The method we apply is the one described in~\cite[Section 4]{ncconfigsFN} or, more generally,  in Sections VI and VII of ~\cite{ancombFS}.  The main idea is that the singularity of $F_T(z)$ with smallest modulus (the \emph{dominant} singularity) and the behaviour around this singularity determine the asymptotic behaviour of the coefficients of $F_T(z)$. More concretely, we summarize in the following theorem the results from~\cite{ancombFS} that are sufficient in our setting (see Lemma~VII.3 and Theorems VI.1 and VI.3 for more details). The subindices indicate derivatives with respect to that variable.

\begin{theorem}\label{thm:asymptotics}
Let $Y(z)$ be defined by the implicit equation $G(z,Y(z))=0$ and let $\rho$ be the dominant singularity of $Y(z)$. 
Suppose that $G(z,y)$ is a polynomial and that $\tau$ is such that $G(\rho, \tau)=0$, $G_z(\rho, \tau)\neq 0$, $G_y(\rho, \tau)=0$ and $G_{yy}(\rho,\tau)\neq 0$. Then
$$[z^n]Y(z)=\gamma \rho^{n} n^{-3/2} \left( 1+O( n^{-1/2})\right) \qquad with \ \gamma=\sqrt {\frac{\rho G_z(\rho, \tau)}{2\pi G_{yy}(\rho,\tau)}}.$$ 
\end{theorem}

Now, to  find the  dominant singularity $\rho$ we use the fact that it must be one of the roots of the discriminant of the equation satisfied by $F_T(z)$, as explained in~\cite[Section VII.7.1]{ancombFS}. Also, Pringsheim's theorem (\cite[Theorem IV.6]{ancombFS}) asserts that $\rho$ is real positive.

Carrying out the calculations for the equations given for the different classes of forests proves the parts in Propositions~\ref{prop:numberforests} and~\ref{prop:forestforest} that do not deal with the number of components.  We give the details for the case of marked forests without isolated vertices given by equation~(\ref{eq:f1s}).

The corresponding discriminant is
$$-16z^3(1-20z+67z^2)^2(4778z^5+7640z^4+793z^3-1454z^2+283z-16).$$
Its real positive roots are $(10-\sqrt{33})/67=0.063513\ldots $ and $(10-\sqrt{33})/67=0.234993\ldots $ from the quadratic factor,  and one root $0.116527\ldots$ of the degree 5 factor. From these candidates, we need to identify which one is the smallest singularity. We follow the methods of~\cite[Section VII.7.1]{ancombFS} to study the behaviour of $F_{T_1^*}(z)$ in a neighbourhood of each of the candidate singularities, starting with the smallest one. It turns out that at $(10-\sqrt{33})/67$ there is a multiple point, but no singularity, whereas the next candidate $0.1165\ldots$ is indeed a singularity. Now we solve equation~(\ref{eq:f1s}) for $z=\rho=0.1165\ldots$; of the four solutions, one is double, which corresponds to the value at the singular point. So $\tau=1.1537\ldots$. Now it is only a matter of checking the hypotheses of Theorem~\ref{thm:asymptotics}.

\bigskip


We need next to refine our generating functions in order to take into account the number of components. We consider the bivariate generating function
$$
F_T(z)=\sum_{n\geq 0} f_{n,k}z^nw^k,$$
where $f_{n,k}$ stands for the number of forests in $\mathcal{F_T}$ with $n$ vertices and $k$ connected components. It is easy to see that equation~(\ref{eq:tf}) becomes
\begin{equation*}
F_T(z,w)=1+wT(zF_T(z,w)).
\end{equation*}

Let $X_{n,k}=[z^nw^k]F_T(z,w)/[z^n]F_T(z)$, that is, the probability that a uniformly chosen forest in $\mathcal{F}_{\mathcal{T}}$ with $n$ vertices has $k$ components. As explained in~\cite[Section 5]{ncconfigsFN} and~\cite[Section IX.7]{ancombFS}, the singular behaviour of $F_T(z,w)$ gives information about the probability $X_{n,k}$. More concretely, for the generating functions we are considering, the singularity $\rho$ of $F_T(z)$  ``lifts'' to a singularity $\rho(w)$ of $F_T(z,w)$ (that means, in particular,  $\rho(1)=\rho$). It can be shown in this situation that the mean of $X_{n,k}$ is $\kappa n+ O(1)$, where $\kappa=-\rho'(1)/\rho$. Moreover, $X_{n,k}$ converges in law to a Gaussian law. This implies  that for each positive $\varepsilon$, $1/2-\varepsilon$ of the forests in $\mathcal{F_T}$ with $n$ vertices have at least $\mu_n n$ components, for sufficiently large $n$.

Again, we provide some detail for the case of marked forests without isolated vertices. The generating function $Y=F_{T_1^*}(z,w)$ satisfies the equation
\begin{align*}
27z(zw+1)^3Y^4-(2z^3w^3+81z^3w^2+18z^2w^2+162z^2w+12zw+81z+4)Y^3+& \\ (18z^2w^2+81z^2w+zw^2+24zw+81z+12)Y^2-(zw^2+12zw+27z+12)Y+4&=0.
\end{align*}
The discriminant of this equation (with respect to $Y$) is a polynomial in $z$ and $w$; it has one double factor which is quadratic in $z$  and a factor that has degree 5 in $z$. Of the five roots of this last factor, we consider the one that at $w=1$ gives the known value of $\rho$, and find $\rho'(1)$ by differentiating with respect to $w$ the degree 5 factor of the discriminant.

Similar calculations give the remaining statements from Propositions~\ref{prop:numberforests} and~\ref{prop:forestforest}. We just mention that for the case of forests where only some of the non-isolated components have marks, the equation that gives the bivariate generating function is
$$F_{T\cup T_1^*}(z,w)=1+T(zF_{T\cup T_1^*}(z,w))+ wT_1^*(zF_{T\cup T_1^*}(z,w)),$$
as we only want to estimate the number of  components that have a mark and not the total number of components.  

\section{Relationship between growth constants}\label{sec:growth}

In this section we explain in a simple way the relationship between growth constants observed at the end of Section~\ref{sec:lbound}, and also prove that the average numbers of components of the corresponding forests are also related. Moreover, we use similar ideas to prove a relationship between the growth constants of nc-forests in the double chain, allowing or not isolated vertices. This relationship enables us to improve the bound $\Omega(13.40^N)$ of Theorem~\ref{thm:lboundforest} to $\Omega(13.61^N)$.

We first focus on nc-forests on points in convex position, allowing or not isolated vertices. Let $f(n)$ 
be the number of forests on $n$ points in convex position, and let $\tilde{f}(n)$ be the number of such forests that have no isolated vertices.  Although we actually know the values of the growth constants, we assume in this section no knowledge about the numbers $f(n)$ and $\tilde{f}(n)$. So first we show that they behave asymptotically as exponentials. For this we use the following lemma on superadditive functions (see~\cite[Lemma 11.6]{coursecombvLW}).

\begin{lemma}\label{lem:limit}
Let $g:\mathbb{N}\rightarrow \mathbb{N}$ be such that $g(i+j)\geq g(i)g(j)$ for all $i,j\in \mathbb{N}$. Then $\lim_{n\rightarrow \infty}g(n)^{1/n}$ exists.
\end{lemma}
 
It is clear that $f$ and $\tilde{f}$  satisfy the hypotheses of this lemma, as the union of a forest on the points $p_1,\ldots, p_i$ and a forest on the points $p_{i+1}, \ldots, p_{i+j}$ gives a forest on $i+j$ points.
Now observe that any nc-forest can be constructed by choosing first some vertices to be isolated and then choosing a nc-forest without isolated vertices; we thus have the relation
$$f(n)=\sum_{i=0}^n \binom{n}{i}\tilde{f}(n-i),$$
where we set $\tilde{f}(0)=1$ for convenience. 
The analysis of this equation gives the  relationship between the growth constants, as we next show. Since there is nothing particular to forests in the argument, we state it in general terms.

\begin{proposition}\label{prop:growth}
Let $f:\mathbb{N}\rightarrow \mathbb{N}$ and $\tilde{f}:\mathbb{N}\rightarrow \mathbb{N}$ be such that there exist constants $b$ and $c$ such that $\lim_{n\rightarrow \infty}f(n)^{1/n}=b$ and $\lim_{n\rightarrow \infty}\tilde{f}(n)^{1/n}=c$.
If for all $n\geq 1$ it holds that 
$$f(n)=\sum_{i=0}^n \binom{n}{i}\tilde{f}(n-i),$$
then $b=1+c$. 
\end{proposition}

\begin{proof}
 Let us start by writing $f(n)=h(n)b^n$ and $\tilde{f}(n)=\tilde{h}(n)c^n$ for some functions $h:\mathbb{N}\rightarrow \mathbb{R}$ and $\tilde{h}:\mathbb{N}\rightarrow \mathbb{R}$ such that $\lim_{n\rightarrow \infty}h(n)^{1/n}=\lim_{n\rightarrow \infty}\tilde{h}(n)^{1/n}=1$. 
Then,
\begin{equation}\label{eq:h}h(n)b^n=\sum_{i=0}^n \binom{n}{i}\tilde{h}(n-i)c^{n-i}.
\end{equation}
We now give upper and lower bounds for the right-hand side that imply $b=c+1$. For the upper bound, take $j$  such that $\tilde{h}(n-j)\geq \tilde{h}(n-i)$ for all $i\leq n$ (note that $j$ depends on $n$). 
Then
$$  \sum_{i=0}^n \binom{n}{i}\tilde{h}(n-i)c^{n-i}\leq \tilde{h}(n-j)\sum_{i=0}^n \binom{n}{i}c^{n-i}=\tilde{h}(n-j) (c+1)^n.$$
By taking the limit of the $n$-th roots of both sides of the inequality $h(n)b^n\leq \tilde{h}(n-j) (c+1)^n$ we immediately get $b\leq c+1$.

For the lower bound, fix any $\varepsilon>0$. Then  there is  $N$ such that  $\tilde{h}(n)\geq (1-\varepsilon)^n$ for all $n\geq N$. For $\alpha$ with $0<\alpha<1$ and  $n\geq N/(1-\alpha)$ we bound the sum in~(\ref{eq:h}) by the term corresponding to $i=\alpha n$, resulting in 
$$\sum_{i=0}^n \binom{n}{i}\tilde{h}(n-i)c^{n-i}\geq \binom{n}{\alpha n} (c(1-\varepsilon))^{n-\alpha n}.$$ 
Using the entropy function as in Section~\ref{sec:lbound}, we have $$\lim_{n\rightarrow \infty}\left(\binom{n}{\alpha n} (c(1-\varepsilon))^{n-\alpha n}\right)^{1/n}=2^{H(\alpha)}(c(1-\varepsilon))^{1-\alpha},$$
and thus $b\geq 2^{H(\alpha)}(c(1-\varepsilon))^{1-\alpha}$ for all $\alpha$. The maximum is achieved at $\alpha=(1+c(1-\varepsilon))^{-1}$ with value $1+c(1-\varepsilon)$. As this holds for all $\varepsilon>0$, we conclude that $b\geq 1+c$, as needed. 
\end{proof}

The relationship between the growth constants in items (ii) and (iii) from Proposition~\ref{prop:numberforests} also follows from the proposition above. Based on results from~\cite{ncconfigsFN}, we immediately get that the growth constant for non-crossing graphs without isolated vertices is $5+4\sqrt{2}$ and that the growth constant for non-crossing partitions without singleton sets is $3$.

\bigskip

There is also a relationship between the average number of components in forests and forests without isolated vertices. Let $\mathcal{F}_n$ and $\tilde{\mathcal{F}}_n$ be the sets of nc-forests with $n$ vertices and of nc-forests with $n$ vertices, none of them being isolated. Given any forest $F$, write $k(F)$ for the number of components of $F$, and let $\mu_n$ and $\tilde{\mu}_n$ denote the mean of the number of components in forests from $\mathcal{F}_n$ and $\tilde{\mathcal{F}}_n$, respectively. Then
$$\mu_n=\frac{\sum_{F\in \mathcal{F}_n}k(F)}{f(n)}=\frac{1}{f(n)}\sum_{i=0}^n\binom{n}{i}\sum_{\tilde{F}\in \tilde{\mathcal{F}}_{n-i}} (k(\tilde{F})+i) =\frac{1}{f(n)}\sum_{i=0}^n \binom{n}{i} \tilde{f}(n-i) (\tilde{\mu}_{n-i}+i). $$
As in the case of the growth constants, this equation determines the limit of $\mu_n/n$, provided we assume that $\lim_{n\rightarrow \infty} \tilde{\mu}_n/n$ exists.

\begin{proposition}\label{prop:mu}
Let $f$, $\tilde{f}$, $b$, and $c$, be as in Proposition~\ref{prop:growth}, and let $(\mu_n)$ and $(\tilde{\mu}_n)$ be real-valued sequences such that
\begin{equation}\label{eq:mu}
\mu_n=\frac{1}{f(n)}\sum_{i=0}^n \binom{n}{i} \tilde{f}(n-i) (\tilde{\mu}_{n-i}+i). 
\end{equation}
If there exists a constant $\tilde{\mu}$ such that $\lim_{n\rightarrow \infty} \tilde{\mu}_n/n=\tilde{\mu}$, then 
$$\lim_{n\rightarrow \infty} \frac{\mu_n}{n}= \frac{1+c\tilde{\mu}}{b}.$$
\end{proposition}

\begin{proof}
We first rewrite the right-hand side of~(\ref{eq:mu}). Easy manipulation gives that the term $\sum_{i=0}^n \binom{n}{i} \tilde{f}(n-i) i$ equals $n\sum_{i=0}^{n-1}\binom{n-1}{i}\tilde{f}(n-1-i)=nf(n-1)$. As for the term $\sum_{i=0}^n \binom{n}{i} \tilde{f}(n-i) \tilde{\mu}_{n-i}$, it can be rewritten as
$$n\sum_{i=0}^{n-1}\binom{n-1}{i}\tilde{f}(n-1-i)\frac{\tilde{f}(n-i)}{\tilde{f}(n-1-i)}\frac{\tilde{\mu}_{n-i}}{n-i}.$$ 

 We now compute $\lim_{n\rightarrow \infty} \mu_n/n$. From the calculations above,
$$\lim_{n\rightarrow \infty} \frac{\mu_n}{n}=\frac{1}{b}+\lim_{n\rightarrow \infty}\frac{1}{f(n)}\sum_{i=0}^{n-1}\binom{n-1}{i}\tilde{f}(n-1-i)\frac{\tilde{f}(n-i)}{\tilde{f}(n-1-i)}\frac{\tilde{\mu}_{n-i}}{n-i}.$$
As
 $$\lim_{n\rightarrow \infty} \frac{\tilde{f}(n)}{\tilde{f}(n-1)}\frac{\tilde{\mu}_{n}}{n}=c\tilde{\mu},$$ for every $\varepsilon>0$ there exists $N$ such that 
$$ c\tilde{\mu}-\varepsilon< \frac{\tilde{f}(n)}{\tilde{f}(n-1)}\frac{\tilde{\mu}_{n}}{n}<c\tilde{\mu}+\varepsilon \mbox{ for } n\geq N.$$
Therefore,
\begin{align}
\frac{1}{f(n)}\sum_{i=0}^{n-1}\binom{n-1}{i}\tilde{f}(n-1-i)\frac{\tilde{f}(n-i)}{\tilde{f}(n-1-i)}\frac{\tilde{\mu}_{n-i}}{n-i}& \leq (c\tilde{\mu}+\varepsilon)\frac{1}{f(n)}\sum_{i=0}^{n-N}\binom{n-1}{i} \tilde{f}(n-1-i)+\notag \\
& \frac{1}{f(n)}\sum_{i=n-N+1}^{n}\binom{n-1}{i}\tilde{f}(n-i)\frac{\tilde{\mu}_{n-i}}{n-i}. \label{ineq}
\end{align}
We claim that the right-hand side tends to $(c\tilde{\mu}+\varepsilon)/b$. First notice that 
the numerator of the second summand  is a sum of $N$ terms that is easily seen to be $O(n^N)$. Similarly, in the first summand the sum of the missing terms for $i$ from $n-N+1$ to $n$ is also $O(n^N)$. Hence, the right-hand side of inequality~(\ref{ineq}) is 
$$(c\tilde{\mu}+\varepsilon)\frac{1}{f(n)}\left( \sum_{i=0}^{n}\binom{n-1}{i} \tilde{f}(n-1-i)+O(n^N)\right)=(c\tilde{\mu}+\varepsilon)\frac{f(n-1)}{f(n)}+
(c\tilde{\mu}+\varepsilon)\frac{O(n^N)}{f(n)},$$
which tends to $(c\tilde{\mu}+\varepsilon)/b$ as desired. 

Thus,
$$\lim_{n\rightarrow \infty} \frac{\mu_n}{n}\leq \frac{1}{b}+\frac{c\tilde{\mu}+\varepsilon}{b},$$
for all $\varepsilon>0$. As an analogous argument shows that
$$\lim_{n\rightarrow \infty} \frac{\mu_n}{n}\geq \frac{1}{b}+\frac{c\tilde{\mu}-\varepsilon}{b}$$
for all $\varepsilon>0$, we conclude that $\lim_{n\rightarrow \infty} \mu_n/n$ exists and equals $(1+c\tilde{\mu})/b.$ 
\end{proof}

The reader can check that this relationship between $\mu$ and $\tilde{\mu}$ holds for the numbers given in Proposition~\ref{prop:numberforests}.

\bigskip

We now turn our attention to the number of nc-forests of the double chain,  allowing or not isolated vertices. We cannot prove that the respective growth constants differ exactly by one (although we believe this is the case), but only that they differ by at least one, which is enough to obtain a lower bound better than the one in Theorem~\ref{thm:lboundforest}.

\begin{proposition}\label{prop:growthconstants}
Let $g(n)$ be the number of nc-forests of the double chain with $2n$ points, and let $\tilde{g}(n)$ be the number of those that have no isolated vertices. There exist constants $\gamma$ and $\tilde{\gamma}$ such that $\lim_{n\rightarrow \infty} g(n)^{1/n}=\gamma^2$ and $\lim_{n\rightarrow \infty} \tilde{g}(n)^{1/n}=\tilde{\gamma}^2$, and these constants satisfy $\gamma\geq\tilde{\gamma}+1$.
\end{proposition}

\begin{proof}
That the growth constants exist is an immediate consequence of Lemma~\ref{lem:limit}. We claim the following relation
\begin{equation}\label{eq:g}
g(n)\geq \sum_{i=0}^n\binom{n}{i}^2\tilde{g}(n-i).
\end{equation}
Indeed, the right-hand side counts the number of nc-forests in the double chain that have the same number of isolated vertices in each chain. As before, write $\tilde{g}(n)=\tilde{h}(n)\tilde{\gamma}^n$ with $\lim_{n\rightarrow \infty} \tilde{h}(n)^{1/n}=1$. For every $\varepsilon>0$ there exists $N$ such that $\tilde{h}(n)\geq (1-\varepsilon)^n$ if $n\geq N$. Thus, for any $\alpha$ and $n\geq N/(1-\alpha)$ we have the lower bound
$$\sum_{i=0}^n\binom{n}{i}^2\tilde{g}(n-i)\geq \binom{n}{\alpha n}^2(1-\varepsilon)^{n-\alpha n} \tilde{\gamma}^{2(n-\alpha n)}.$$
The $n$-th root of the left-hand side tends to $2^{2H(\alpha)}(\tilde{\gamma}(1-\varepsilon))^{1-\alpha}$, which is maximized at $\alpha=(1+\tilde{\gamma}(1-\varepsilon))^{-1}$ with value $(1+\tilde{\gamma}(1-\varepsilon))^2$. As this holds for all $\varepsilon >0$, we conclude that $\gamma\geq \tilde{\gamma}+1$.
\end{proof}

We remark that $\left(\sum_{i=0}^n\binom{n}{i}^2\tilde{\gamma}^{2(n-i)}\right)^{1/n}$ tends indeed to $(1+\tilde{\gamma})^2$. For this,  we write the sum as  an evaluation of the Legendre polynomial $P_n(x)=\sum_{k=0}^n\binom{n}{k}^2(x-1)^{n-k}(x+1)^k/2^n $~(\cite{AS}), that is,
$$\sum_{i=0}^n\binom{n}{i}^2\tilde{\gamma}^{2(n-i)}=(1-\tilde{\gamma}^2)^nP_n\left( \frac{1+\tilde{\gamma}^2}{1-\tilde{\gamma}^2}\right).$$
As the generating function $\sum_{n\geq 0} P_n(x)t^n  $ is $(1-2tx +t^2)^{-1/2}$ and the inverse of the dominant singularity of the generating function gives the exponential growth of its coefficients (see~\cite[Section IV.3.2]{ancombFS}), for a fixed value of $x$ we have $\lim_{n\rightarrow \infty} P_n(x)^{1/n}=(x-\sqrt{x^2-1})^{-1}=x+\sqrt{x^2-1}$, which for $x=(1+\tilde{\gamma}^2)(1-\tilde{\gamma}^2)$ gives the claimed result.

We actually believe that relation~(\ref{eq:g}) is asymptotically an equality, and we conjecture that $\gamma=\tilde{\gamma}+1$.  

\bigskip

We can apply Proposition~\ref{prop:growthconstants} by taking the number of nc-trees as a lower bound for the number of nc-forests without isolated vertices. Thus, from Theorem~\ref{thm:lbound} we immediately get that there are $\Omega(13.52^N)$ nc-forests in the double chain. We can do slightly better by using a construction similar to the ones in Section~\ref{sec:lbound}. 

\begin{corollary}
The double chain on $N$ points has $\Omega(13.61^N)$ non-crossing forests.
\end{corollary}

\begin{proof}
We modify the construction that proves the bound in Theorem~\ref{thm:lboundforest} so that no vertex in the resulting forest is isolated. For this, take
\begin{itemize}
\item[c1)] a set of $\ell$ vertices from the upper chain, all of them marked,
\item[c2)] a nc-forest on $n-\ell$ points in convex position such that none of them is isolated and where at least $k-\ell$ components are marked,
\item[c3)] a subset of $k$ vertices among the $\ell$ from c1) and the marks from c2),
\item[c4)] a nc-forest $F_L$ on $n$ points in convex position, and
\item[c5)] a subset $M_L$ of $k$ vertices from $F_L$.
\end{itemize}
Then do the following to obtain a forest without isolated vertices:
\begin{itemize}
\item[f1)] Match the $k$ vertices from c3) with the ones from c5);
\item[f2)] join the other marks from c1) and c2) to the leftmost visible vertex on $M_L$;
\item[f3)] join each vertex in $F_L$ that still is isolated with the rightmost visible mark on the upper chain.
\end{itemize}

We now need to compute the number of nc-forests on $n$ points in convex position required in item c2), and also the average number of marked components. This is done as in Section~\ref{sec:gfs}, giving that the growth constant for the number of such forests is $8.8643\ldots$ and that the average number of marked  components is $\kappa n$ with $\kappa= 0.1231\ldots$, the limiting distribution being again Gaussian. Actually, the value of the growth constant  follows immediately from Proposition~\ref{prop:growth}, but observe that the average number of components needs to be computed with analytic methods, as relation~(\ref{eq:mu}) does not hold in this case.

Now, setting $k=\alpha n$ and $\ell=\beta n$, we conclude that the number of nc-forests of the double chain without isolated vertices is at least
$$\binom{n}{\beta n} 8.8643^{n-\beta n} \binom{0.1231(n-\beta n)+\beta n}{\alpha n}8.2246^n \binom{n}{\alpha n},$$
which for $\alpha=0.263$ and $\beta=0.267$ gives $\Omega^*(12.6108^N)$. Finally, from Proposition~\ref{prop:growthconstants} we conclude that the double chain has $\Omega^*(13.6108^N)$ nc-forests.
\end{proof}

Another consequence of Proposition~\ref{prop:growthconstants} is that the number of nc-spanning trees of the double chain is $O(23.68^N)$. Indeed, in~\cite{D11} it is proved that the corresponding number of nc-forests is $O(24.68^N)$, so the growth constant for spanning trees is at least one less. In the following section we improve this upper bound on the number of nc-spanning trees to $O(22.12^N)$.

\section{The upper bound}{\label{sec:upper}}

Recall that for any nc-spanning tree of the double chain, the vertices on the upper and lower chains induce two forests $F_U$ and $F_L$ on a set of $n$ points in convex position and a forest $F_I$ formed by interior edges, i.e., edges with one endpoint on each chain.
We first count the number of forests $F_I$ with a given number of edges.

\begin{proposition}{\label{prop:countInterior}}
Let $F_I(N,k)$ be the set of nc-forests in the interior of the double chain on $N=2n$ vertices with $k$ edges.
The number of forests $|F_I(N,k)|$ is
\begin{equation}\label{eq:forests}
|F_I(N,k)|=\sum_{\ell=1}^{\min(k,n)}\sum_{i=1}^{n-\ell+1}\sum_{j=1}^{n-\ell+1}{{{n-i}\choose{\ell-1}} {{n-j}\choose{\ell-1}} {{2n-\ell-i-j+1}\choose{k-\ell}}  }.
\end{equation}
\end{proposition}

\begin{proof}
We assign each forest of $F_I(N,k)$ to a unique matching described by the following algorithm:
Scan the edges of the forest from left to right. 
Let $u_1$ be the vertex on the upper chain incident to the first edge encountered. The rightmost edge incident to $u_1$ is the first edge of the matching. Let $d_1$ be the other endpoint of this first edge of the matching on the lower chain. 
Assume the first $r \geq 1$ edges of the matching are determined and let $d_r$ be the endpoint on the lower chain of the last edge of the matching encountered so far. Let $u_{r+1}$ be the vertex on the upper chain incident to the next edge of the forest which does not have $d_r$ as its other endpoint. Then the rightmost edge incident to $u_{r+1}$ is the next edge of the matching. The algorithm stops when all the edges of the forest are scanned. The resulting matching has between $1$ and $n$ edges.  For the forest in Figure~\ref{fig:lbtree} the matching corresponds to the bold edges (the edges of $M_2$). 

Now, to count  $|F_I(N,k)|$, we first consider all matchings $M_2$ with a fixed number $\ell$ of edges, and  then sum over all possible values of $\ell$.
We extend each matching $M_2$ in all possible ways to a forest $F_I(M_2)$ which has $M_2$ as its uniquely assigned matching. This can be done by scanning the interior of the double chain from left to right again. Denote the $\ell$ edges of $M_2$ by $u_id_i$, $1 \leq i \leq \ell;$ $u_i$ is the vertex on the upper chain.  First examine the points of the double chain to the left of $u_1d_1$, the first edge of $M_2$. The vertices on the upper chain before $u_1$ can not have incident edges in $F_I(M_2)$. The vertices on the lower chain before $d_1$ can (optionally) be connected to $u_1$. Assume we have scanned all points to the left of the edge $u_rd_r$, $r \geq 1$.   Then $d_r$ can (optionally) be connected to all vertices on the upper chain from $u_r$ to $u_{r+1}.$ And $u_{r+1}$ can (optionally) be connected to all vertices from $d_r$ to $d_{r+1}$ on the lower chain. We continue in this way until we reach the last edge  $u_\ell d_\ell$ of $M_2$. Then all the vertices on the upper chain on the right of $u_\ell$ can (optionally) be connected to $d_\ell$, but the vertices on the lower chain on the right of $d_\ell$ can not have incident edges. Thus, we can build a forest with $k$ edges $F_I(M_2)$, assigned to $M_2$, by adding a subset of $k-\ell$ edges from these optional edges to $M_2$. Altogether there are $2n-1-\ell-(i-1)-(j-1)$ optional edges, where $2n-1$ is the number of edges of a spanning tree of the double chain, the matching has $\ell$ edges, $i-1$ points are on the upper chain to the left of $u_1$ and $j-1$ points are on the lower chain to the right of $d_\ell.$ This gives the factor ${{2n-\ell-i-j+1}\choose{k-\ell}}$ of Equation~(\ref{eq:forests}). The terms ${{n-i}\choose{\ell-1}}$ and ${{n-j}\choose{\ell-1}}$ correspond to the selection of points for building a matching $M_2$ with $\ell$ edges.   
\end{proof}

\begin{theorem}
The double chain on $N$ points has $O(22.12^N)$ non-crossing spanning trees.\end{theorem}

\begin{proof}
The product of the numbers of forests $F_U$, $F_L$ and $F_I$  gives an upper bound on the number of nc-spanning trees of the double chain. This asymptotic counting can be refined by only counting spanning trees with forests from the set $F_I(N,k)$ for a certain value of $k$. Still, we also will count graphs that contain cycles. 

First, we partition the set of nc-spanning trees of the double chain into $2n-1$ classes, according to the number of edges of $F_I$.
For asymptotic counting, it is sufficient to only consider the one class of nc-spanning trees, with $k$ edges in forest $F_I$, that contains most spanning trees, over-counting by a factor of at most $2n-1$.
Given a non-crossing spanning tree, removing the $k$ edges of the forest $F_I$ in the interior, the graph breaks into a forest consisting of $k+1$ components; a forest on the lower chain of $k_1$ components, and a forest on the upper chain of $k_2$ components, with $k_1+k_2=k+1.$ 
It is sufficient to only count the number of spanning forests for a certain value of $k_1$, that maximizes the product of the three terms. This is, again, because we can ignore polynomial factors. For the same reason, we also can simplify some terms in the following, neglecting some constants. In particular we can assume that the forests on the two chains have $k$ components instead of $k+1,$ and set  $k=\alpha n$ and $k_1= \beta n$, with $0<\alpha<2$ and $0<\beta < \min\{\alpha,1\}.$    
We shall show that the maximum is attained for $k_1=k_2=\frac{k}{2}$.\\

The number of nc-forests with $c$ components on a set of $n$ points in convex position is given by the formula~\cite{ncconfigsFN} 
$$F_{n,c}=\frac{1}{2n-c}{{n}\choose{c-1}}{{3n-2c-1}\choose{2n-c-1}}.$$ 
In our setting, once fixed $k$, the factor for the forests in the interior is independent from the other two factors. We get for the product of the number of forests in the two chains
$$f_{n,k_1,k_2}=\frac{1}{2n-k_1}{{n}\choose{k_1-1}}{{3n-2k_1-1}\choose{2n-k_1-1}} \frac{1}{2n-k_2}{{n}\choose{k_2-1}}{{3n-2k_2-1}\choose{2n-k_2-1}}.$$
Then, 
 $$f_{n,k_1,k_2} \approx {{n}\choose{\beta n}}{{3n-2\beta n}\choose{2n-\beta n}}
  {{n}\choose{\alpha n - \beta n}}{{3n-2(\alpha n - \beta n)}\choose{2n-(\alpha n - \beta n)}}.$$

As in Section~\ref{sec:lbound} we use the binary entropy function to estimate a binomial coefficient, which gives 
$$f_{n,k_1,k_2}  \approx 2^{\left( H(\beta)+(3-2\beta)H(\frac{2-\beta}{3-2\beta})+H(\alpha-\beta)+(3-2\alpha+2\beta)H(\frac{2-\alpha+\beta}{3-2\alpha+2\beta})      \right)n}.$$
For $n$ and $\alpha$ fixed, this is a function of $\beta$ we want to maximize. Equivalently we can maximize the logarithm of basis $2$ of this function and we can ignore the factor $n$. Hence, we maximize  $$g_{\alpha}(\beta)=H(\beta)+(3-2\beta)H(\frac{2-\beta}{3-2\beta})+H(\alpha-\beta)+(3-2\alpha+2\beta)H(\frac{2-\alpha+\beta}{3-2\alpha+2\beta}).$$ 
Then standard calculations show that $g_{\alpha}(\beta)$ attains its maximum at $\beta=\frac{\alpha}{2}$.\\

Next, we determine the asymptotic growth of $|F_I(N,k)|.$ Ignoring polynomial factors, Equation~(\ref{eq:forests}) can be bounded from above by
$$\sum_{\ell=1}^{\min(k,n)}{{n \choose \ell}^2 {{2n-\ell}\choose k-\ell}}.$$ 
In fact, for asymptotic counting it is sufficient to only consider the largest term in this sum. We determine the value of $\ell$ that gives the largest summand. We set $k=\alpha n$ and $\ell= \lambda n$, with $0<\alpha < 2$ and $0<\lambda < 1.$
Then a summand has the form 
$$\approx {{n \choose \lambda n}^2 {{(2-\lambda)n}\choose (\alpha -\lambda)n}} \approx 2^{\left( 2H(\lambda)+(2-\lambda)H(\frac{\alpha-\lambda}{2-\lambda}) \right)n}.$$
For $\alpha$ fixed, we maximize 
$$h_{\alpha}(\lambda)=2H(\lambda)+(2-\lambda)H(\frac{\alpha-\lambda}{2-\lambda}).$$
Solving $h_{\alpha}'(\lambda)=0$, we obtain
$$\lambda=\frac{1+2\alpha\pm\sqrt{1+4\alpha}}{2\alpha}$$ and verify that the maximizing term is indeed 
$h_\alpha(\frac{1+2\alpha-\sqrt{1+4\alpha}}{2\alpha})$.\\


We finally can write the product of the three terms (number of forests in lower chain, number of forests in upper chain, number of forests in the interior) as a function of $n$ and $k=\alpha n$. 
We already have seen that the first two terms are equal.
The number of non-crossing trees of the double chain on $2n$ points is asymptotically  bounded from above  by  
$$
2^{t(\alpha)n}=2^{\left(2(H(\frac{\alpha}{2})+(3-\alpha)H(\frac{2-\alpha/2}{3-\alpha})  )    +h_\alpha(\frac{1+2\alpha-\sqrt{1+4\alpha}}{2\alpha})  \right)n}.$$
The function $t(\alpha)$ is maximized for $\alpha=0.750614$ with $2^{t(\alpha)}=8.93341.$ 
Therefore, the number of non-crossing spanning trees of the double chain on a set of $N=2n$ points
is at most $2^{8.93341N/2}=22.1112^N.$\\

\end{proof}

\section*{Acknowledgements}
The first author is supported by Projects  MTM2012-30951, DGR2009-SGR1040, and EuroGIGA, CRP ComPoSe: grant EUI-EURC-2011-4306. The second author is supported by Projects  MTM2011-24097 and DGR2009-SGR1040.


\begin{thebibliography}{99}



\bibitem{AS}
M. Abramowitz, I. Stegun, \emph{Handbook of mathematical functions with formulas, graphs, and mathematical tables}, Dover, New York, 1965.



\bibitem{AOSS}
O. Aichholzer, D. Orden, F. Santos, B. Speckmann,
{\em On the number of pseudo-triangulations of certain point sets.}
Journal of Combinatorial Theory Series A, 115 (2), pp.~254--278, 2008.


\bibitem{ajtai}
M. Ajtai, V. Chv{\'a}tal, M. Newborn, E. Szemer{\'e}di, 
{\em{Crossing-free subgraphs.}}
 Ann. Discrete Math. 12, pp. 9-12, 1982.





\bibitem{D99}
A. Dumitrescu,
{\em On two lower bound constructions.}
Proc. $11^{th}$ Canadian Conference on Computational Geometry,
Vancouver, British Columbia, Canada, pp.~111-114, 1999.

\bibitem{D11}
A. Dumitrescu, A. Schulz, A. Sheffer, C.D. T\'oth,
{\emph{Bounds on the maximum multiplicity of some common geometric graphs.}}
SIAM J. Discrete Math. 27(2), pp.~802--826, 2013.

%


\bibitem{ncconfigsFN} P. Flajolet, M. Noy,
{\it{ Analytic combinatorics of non-crossing configurations.}}
Discrete Mathematics  204, pp.~203--229, 1999.

\bibitem{ancombFS} P.~Flajolet, R.~Sedgewick, \emph{Analytic combinatorics.}
Cambridge University Press, Cambridge, 2009.

\bibitem{GNT}
A. Garc\'ia, M. Noy, J. Tejel, {\emph{Lower bounds on the number of
crossing-free subgraphs of $K_n$.}} Computational Geometry: Theory
and Applications 16, pp. 211-221, 2000.

\bibitem{hoffmann}
M. Hoffmann, A. Schulz, M. Sharir, A. Sheffer, C.D. T\'oth, E. Welzl,
{\emph{Counting plane graphs: flippability and its applications.}}
Thirty Essays on Geometric Graph Theory, Springer, 2012.


\bibitem{eurocomb}
C. Huemer, A. de Mier, \emph{An improved lower bound on the maximum number of non-crossing spanning trees.}
Proc. $7^{th}$ European Conference on Combinatorics, Graph Theory and Applications, Pisa, Italy, pp. 283--290, 2013.

\bibitem{coursecombvLW} J. H. van Lint,  R. M. Wilson, \emph{A course
in Combinatorics}, second edition, Cambridge University Press (2001).
















\end{thebibliography}
\end{document}